\documentclass[12pt]{article}

\usepackage{amsmath,amsfonts,amssymb,amsthm,amscd}

\title{On Dimensions, Standard Part Maps, and $p$-Adically Closed Fields}
\date{\today}

\author{Ningyuan Yao\\Fudan University}

\newtheorem{Thm}{Theorem}[section]
\newtheorem{Prop}[Thm]{Proposition}
\newtheorem{Def}[Thm]{Definition}
\newtheorem{Rmk}[Thm]{Remark}
\newtheorem{Lemma}[Thm]{Lemma}
\newtheorem{Cor}[Thm]{Corollary}
\newtheorem{Fact}[Thm]{Fact}

\newtheorem*{Claim1}{Claim 1}
\newtheorem*{Claim2}{Claim 2}
\newtheorem*{Claim3}{Claim 3}
\newtheorem*{Claim4}{Claim 4}
\newtheorem*{Claim5}{Claim 5}

\newtheorem*{Thm1}{Theorem 1}

\newtheorem*{Thm2}{Theorem 2}
\newtheorem*{Hensel's Lemma}{Hensel's Lemma}

\newtheorem*{Claim}{Claim}

\newcommand{\R}{\mathbb R}
\newcommand{\Q}{{\mathbb Q}_p}
\newcommand{\Z}{{\mathbb Z}_p}

\newcommand{\N}{\mathbb N}

\newcommand{\K}{\mathbb K}

\newcommand{\sq}{\subseteq}

\DeclareMathOperator{\Th}{Th}

\DeclareMathOperator{\Int}{Int}

\DeclareMathOperator{\st}{st}

\DeclareMathOperator{\acl}{acl}

\begin{document}
\maketitle
\begin{abstract} The aim of this paper is to study the dimensions and standard part maps between the field of $p$-adic numbers $\Q$ and its elementary extension $K$ in the language of rings $L_r$.

We show that for any $K$-definable set $X\sq K^m$, $\dim_K(X)\geq \dim_{\Q}(X\cap \Q^m)$. Let $V\sq K$ be convex hull of $K$ over $\Q$, and $\st: V\rightarrow \Q$ be the standard part map.  We show that for any $K$-definable function $f:K^m\rightarrow K$, there is definable subset $D\sq\Q^m$ such that $\Q^m\backslash D$ has no interior, and for all $x\in D$, either $f(x)\in V$ and $\st(f(\st^{-1}(x)))$ is constant, or $f(\st^{-1}(x))\cap V=\emptyset$.

We also prove that $\dim_K(X)\geq \dim_{\Q}(\st(X\cap V^m))$ for every definable $X\sq K^m$.
\end{abstract}

\section{Introduction}
In \cite{L. D-T-convex-II},  L. van den Dries consider a pair $(R,V)$, where ${R}$ is an $o$-minimal extension of a real closed field, and $V$ is a convex hull of an elementary submodel $M$ of $R$. Let  $\mu\sq R$ be the set infinitesimals over $M$ and  $\hat V=V/{\mu}$ be the reside field with residue class map $x\mapsto \hat x$. If $M$ is Dedekind complete in $R$, then $\hat V=M$ and the residue class map coincide the standard part map $\st: R\longrightarrow M$. In this context, van den Dries showed the follows:
\begin{Thm1}\cite{L. D-T-convex-II}
Let $S\sq R^n$ be $R$-definable and $\hat S=\{\hat x| \  x\in S\cap V^n\}$. Then
\begin{itemize}
  \item [(i)] $S\cap M^n$ is definable in $M$ and $\dim_{M}(S\cap M^n)\leq \dim_{R}(S)$;
  \item [(ii)] $\st(S)$ is definable in $M$ and  $\dim_{M}(\st(S))\leq \dim_{ R}(S)$.
\end{itemize}
\end{Thm1}

\begin{Thm2}\cite{L. D-T-convex-II}
Let $f:R^m\rightarrow R$ be an $R$-definable function. Then  here is a finite partition ${\cal P}$ of $M^m$ into definable sets, where each set in the partition is either open in $M^m$ or lacks of interior.  On each open set $C\in \cal P$ we have:
\begin{itemize}
\item [(i)] either $f(x)\notin V$ for all $x\in C^h$;
  \item [(ii)] or there is a continuous function $g: C\longrightarrow M$, definable in $M$, such that $f(x)\in V$ and $\st(f(x))=g(\st(x))$, for all $x\in C^h$,
\end{itemize}
where $C^h$ is the hull of $C$ defined by
\[
C^h=\{\bar x\in R^m| \exists \bar y\in C\big( \bigwedge_{i=1}^m(x_i-y_i\in \mu)\big)\}.
\]

\end{Thm2}
\begin{Rmk}
For any topological space $Y$, and $X\sq Y$, by $\Int(X)$ we mean  the set of interiors in $X$. Namely, $x\in \Int(X)$ iff there is an open neighborhood $B\sq Y$ of $x$ contained in $X$.
\end{Rmk}
There are fairly good analogies between the field of reals $\R$ and the field of $p$-adic numbers $\Q$, in both model-theoretic and field-theoretic view. For example, both of them are complete and  locally compact topological fields, are distal and dp-minimal structures,  have quantifier eliminations with adding the new predicates for $n$-th power, and have cell decompositions.

In this paper, we treat the $p$-adic analogue of  above two Theorems, where $M$ is replaced by $\Q$, and $R$ is replaced by an arbitrary elementary extension $K$ of $\Q$. In our case, the convex hull $V$ is the set
\[\bigg\{x\in K|\ x=0\vee \exists n\in \mathbb Z  \bigg(v(x)>n\bigg)\bigg\}\]
and $\mu$, the infinitesimals of $K$ over $\Q$, is the set
\[\bigg\{x\in K|\ x=0\vee \forall n\in \mathbb Z \bigg(v(x)>n\bigg)\bigg\}.\]
By Lemma 2.1 in \cite{GPY}, for every $x\in V$, there is  a unique element $\st(x)$ in $\Q$ such that $a-\st(a)\in \mu$, we call it the standard part of $a$ and $\st:a\mapsto \st (a)$ the standard part map. It is easy to see that $\st: V\longrightarrow \Q$ is a surjective ring homomorphism and $\st^{-1}(0)=\mu$. So $\hat V=V/\mu$ is isomorphic to $\Q$ in our context. With the notations as above, we now highlight our main results.
\begin{Thm}
Let $S\sq K^n$ be $K$-definable. Then
\begin{itemize}
  \item [(i)] $S\cap \Q^n$ is definable in $\Q$ and $\dim_{\Q}(S\cap \Q^n)\leq \dim_{K}(S)$;
  \item [(ii)] $\st(S\cap V^n)$ is definable in $\Q$ and  $\dim_{\Q}(\st(S\cap V^n))\leq \dim_{K}(S)$.
\end{itemize}
\end{Thm}

\begin{Thm}\label{Main-thm-2}
Let $f:K^m\rightarrow K$ be an $K$-definable function. Then here is a finite partition ${\cal P}$ of $\Q$ into definable sets, where each set in the partition is either open in $\Q^m$ or lacks of interior. On each open set $C\in \cal P$ we have:
\begin{itemize}
\item [(i)] either $f(x)\notin V$ for all $x\in C^h$;
  \item [(ii)] or there is a continuous function $g: C\longrightarrow \Q$, definable in $\Q$, such that $f(x)\in V$ and $\st(f(x))=g(\st(x))$, for all $x\in C^h$.
\end{itemize}
\end{Thm}

In the rest of this introduction we give more  notations and model-theoretic approach.

\subsection{Notations}

Let $p$ denote a fixed prime number, $\Q$ the field the $p$-adic field, and $v:\Q\backslash \{0\}\rightarrow \mathbb Z$ is the valuation map. Let $K$ be a fixed elementary extension of $\Q$. Then valuation  $v$  extends to a valuation map from $K\backslash \{0\}$ to $\Gamma_K$, we also denote it by $v$, where $(\Gamma_K, +,<,0)$ is the corresponding elementary extension of $(\mathbb Z, +,<,0)$.

\begin{Fact}
Let $v: K\backslash\{0\}\longrightarrow \Gamma_K$ be as above. Then we have
\begin{itemize}
  \item $v(xy)=v(x)+v(y)$ for all $x,y\in K$;
  \item $v(x+y)\geq min\{v(x),v(y)\}$, and $v(x+y)= min\{v(x),v(y)\}$ if $v(x)\neq v(y)$;
  \item For $x\in \Q$, $|x|=p^{-v(x)}$ if $x\neq 0$ and $|x|=0$ if $x=0$  defines a non-archimedean metric on $\Q$.
  \item $K$ is a non--archimedean topological field.
\end{itemize}
\end{Fact}

We will assume a basic knowledge of model theory. Good references are \cite{P-book} and \cite{M-book}.  We will be referring a lot to the comprehensive survey \cite{Luc Belair} for the basic model theory of the p-adics. A key
point is Macintyre��s theorem \cite{Macintyre} that $\Th(\Q,+,\times, 0,1)$ has quantifier elimination in the language $L=L_r\cup \{P_n|\ n\in \N^+\}$,  where $L_r$ is the language of rings, and  the predicate $P_n$ is interpreted as the $n$-th powers
\[
\{y\in \Q|\ y\neq 0\wedge \exists x(y=x^n)\}
\]
for each $n\in \N^+$. Note that $P_n$ is definable in $L_r$.
Moreover, the valuation is definable in $L_r$ as follows.
\begin{Fact}\cite{Denef-cell-dec}\label{valuation-is-definable}
Let $f,g\in K[x_1,...,x_m]$. Then $\{\bar a\in K^m|\ v(f(\bar a))\leq v(g(\bar a))\}$ is definable.
\end{Fact}

\begin{Rmk}
It is easy to see from Fact \ref{valuation-is-definable} that $\{a\in K| v(a)=\gamma\}$ and $\{a\in K| v(a)<\gamma\}$ are definable for any fixed $\gamma\in \Gamma_K$.
\end{Rmk}

For $A$ a subset of $K$, by an $L_A$-formula we mean a formula with parameters from $A$. By $\bar x,\bar  y, \bar z$ we mean arbitrary $n$-variables and $\bar a,\bar  b, \bar c\in K^n$ denote $n$-tuples in $K^n$ with $n\in \N^+$. By $|\bar x|$, we mean the length of the tuple $\bar x$.
We say that $X\sq K^m$ is $A$-definable  if there is a $L_A$-formula $\phi(x_1,...,x_m)$ such that
\[
X=\{(a_1,...,a_m)\in K^m| \ K\models \phi(a_1,...,a_m)\}.
\]
We also denote $X$ by $\phi(K^m)$ and say that $X$ is defined by $\phi(\bar x)$. We say that $X$ is definable in $K$ if $X\sq K^m$ is $K$-definable.
If $X\in \Q^m$ is defined by some $L_{\Q}$-formula $\psi(\bar x)$. Then by $X(K)$ we mean $\psi(K^m)$, namely, the realizations of $\psi$ in $K$, which is a definable subset of $K^m$.

 For any subset $A$ of $K$, by $\acl(A)$ we mean the algebraic closure of $A$. Namely, $b\in \acl(A)$ if and only if there is a formula $\phi(x)$ with parameters from $A$ such that $b\in\phi(K)$ and $\phi(K)$ is finite.  Let $\alpha=(\alpha_1,...,\alpha_n)\in K^m$, we denote $\acl(A\cup\{\alpha_1,...,\alpha_n\})$ by $\acl(A,\alpha)$.

By a saturated extension $\K$ of $K$, we mean that $|\K|$ is a sufficiently large cardinality, and every type over $A\sq \K$ is realized in $\K$ whenever $|A|< |\K|$.

\subsection{Preliminaries}

The $p$-adic field $\Q$ is a complete, locally compact topological field, with basis given by the sets
\[
B_{(a,n)}=\{x\in \Q|\ x=a \wedge v(x-a)\geq n\}
 \]
for $a\in \Q$ and $n\in \mathbb Z$.  The elementary extension $K\succ \Q$ is also a topological field but not need to be complete or locally compact. Let $X\sq K^m$, we say that $\bar a\in X$  is an interior if there is $\gamma\in \Gamma_k$ such that
\[B_{(\bar b, \gamma)}=\{(b_1,...,b_m)\in \K|\ \bigwedge_{i=1}^{n} v(a_i-b_i)>\gamma \}\sq X.\]

Let $V=\{x\in K|\ x=0\vee \exists n\in \mathbb Z(v(x)>n)\}$. We call $V$ the convex hull of $\Q$. It is easy to see that for any  $a,b\in K$, if $b\in V$ and $v(a)>v(b)$, then $a\in V$. As we said before, for every $a\in V$, there is a unique $a_0\in \Q$ such that $v(a-a_0)>n$ for all $n\in \mathbb Z$. This gives a  map $a\mapsto a_0$ from $V$ onto $\Q$. We call this map the standard part map, denoted by $\st: V\longrightarrow \Q$. For any $\bar a=(a_1,...,a_m)\in V^{m}$, by $\st(\bar a)$ we mean $(\st(a_1),...,\st({a_m}))$. Let $f(\bar x)\in K[\bar x]$ be a polynomial with every coefficient contained in $V$. Then by $\st(f)$, we mean the polynomial over $\Q$ obtained by replace each coefficient of $f$ by its standard part. Let
\[
\mu=\{a\in K|\ \st(a)=0\},
\]
which is the collection of all infinitesimals of $K$ over $\Q$. It is easy to see that for any $a\in K\backslash \{0\}$, $a\notin V$ iff $a^{-1}\in \mu$.

Any definable subset $X\sq K^n$ has a topological dimension which is defined as follows:
\begin{Def}
Let $X\sq K^n$. By $\dim_K(X)$, we mean the maximal $k\leq n$ such that the image of the projection
\[
\pi:X\longrightarrow K^k;\ \ (x_1,...,x_n)\mapsto (x_{r_1},...,x_{r_k})
\]
has interiors, for  suitable  $1\leq r_1<...<r_k\leq n$. We call $\dim_K(X)$ the topological dimension of $X$.
\end{Def}

Recall that $\Q$ is a  geometry structure (see Definition 2.1 and Proposition 2.11 of \cite{Hrushovski-Pillay} ), so any $K\models \Th(\Q)$ is a geometry structure. The fields has geometric structure are
certain fields in which model-theoretic algebraic closure equals field-theoretic algebraic closure.

Every geometry structure is a pregeometry structure, which means that for any $\bar a=(a_1,...,a_n)\in K^n$ and $A\sq K$,  $\dim(\bar a/A)$ makes sense, which by definition is the maximal $k$ such that
$a_{r_{1}}\notin \acl(A)$ and $a_{r_{i+1}}\notin \acl(A, a_{r_{1}},...,a_{r_{i}})$ for
some subtuple $(a_{r_1},...,a_{r_k})$ of $\bar a$.  We call $\dim(\bar a/A)$ the algebraic dimension of $\bar a$ over $A$.

\begin{Fact}\label{alg-dimension and topological dimension}\cite{Hrushovski-Pillay}
Let $A$ be a subset of $K$ and  $X$ an $A$-definable subset of $K^m$.
\begin{itemize}
  \item [(i)] If $\bar a\in K^m$ and $\bar b\in K^n$. Then we have
\[
\dim(\bar a,\bar b/A)=\dim(\bar a/A,\bar b)+\dim(\bar b/A)=\dim(\bar b,\bar a/A).
\]
   \item [(ii)] Let $\K\succ K$  be a saturated model.   Then $\dim_K(X)=\max\{\dim(\bar a/A)|\ \bar a\in X(\K)\}$.
   \item [(iii)] Let $\phi(x_1,...,x_m,y_1,...,y_n)$ be any $L_A$-formula and $r\in \N$. Then the set
  \[
  \{\bar b\in K^n|\ \dim_K(\phi(K^m,\bar b))\leq r\}
  \]
  is $A$-definable.
\item [(iv)] If $X\sq K$ is $K$-definable. Then $X$ is infinite iff $\dim_K(X)\geq1$.
\item [(v)] Let $A_0$ be a countable subset of $\Q$, and let $Y$
be an $A_0$-definable subset of $\Q^n$. Then there is $\bar a_0\in Y$ such that $\dim(\bar a_0/A_0) = \dim_{\Q}(Y)$.
\end{itemize}
\end{Fact}
It is easy to see from Fact \ref{alg-dimension and topological dimension} that for any $L_K$-formula $\phi(x_1,...,x_n)$ and $K'\succ K$, we have
\[\dim_K(\phi(K^n))=\dim_{K'}(\phi({K'}^n)).\]
We will write $\dim_K(X)$ by $\dim(X)$ if there is no ambiguity. If the function $f:X\longrightarrow K$ is definable in $K$, and $Y\sq X\times K$ is the graph of $f$. Then we conclude directly  that  $\dim(X)=\dim(Y)$ by Fact \ref{alg-dimension and topological dimension} (ii).

For later use, we recall some well-known facts and terminology.
\begin{Hensel's Lemma}
Let $\Z=\{x\in \Q| x=0\vee v(x)\geq 0\}$ be the valuation ring of $\Q$. Let $f(x)$ be a polynomial over $\Z$ in one variable $x$, and let $a\in \Z$ such that $v(f(a))>2n+1$ and $v(f'(a))\leq n$, where $f'$ denotes the derivative of $f$. Then there exists a unique $\hat a\in \Z$ such that
\[
f(\hat a)=0\ \text{and}\ v(\hat a-a)\geq n+1.
\]
\end{Hensel's Lemma}

We say a field $E$ is a Henselian field if Hensel's Lemma holds in $E$. Note that to be a henselian field is a first-order property of a field in the language of rings. Namely, there is a $L_r$-sentence $\sigma$ such that $E\models \sigma$ iff $E$ is a henselian field. So any $K\succ \Q$ is henselian.

\section{Main results}
\subsection{Some Properties of Henselian Fields}
Since $\Q$ is complete and local compact, it is easy to see that:
\begin{Fact}\label{Qp is closed in E}
Suppose that $E$ is a finite (or algebraic) field extension of $\Q$. Then for any $\alpha\in E\backslash \Q$, there is $n\in \mathbb Z$ such that $v(\alpha-a)<n$ ($|\alpha-a|>p^{-n}$) for all $a\in \Q$. Namely, $\Q$ is closed in $E$.
\end{Fact}

We now show that Fact \ref{Qp is closed in E} holds for any $K\models \Th(\Q)$.

\begin{Lemma}\label{lim is 0 implies roots exist}
Let $K$ be a henselian field, $R=\{x\in K|\ x=0\vee v(x)\geq 0\}$ be the valuation ring of $K$, and  $f(x)\in R[x]$ a polynomial, $D\sq \Gamma_K$ a cofinal subset, and $X=\{x_d|\ d\in D\}\sq R$. If
\[{\lim_{{d\in D,d \to +\infty}}}f(x_d)=0,\]
Then there exist a cofinal subset $I\sq D$ and $a\in K$ such that
\[{\lim_{{i\in I,i \to +\infty}}}x_i=a \ \text{and}\  f(a)=0.\]
\end{Lemma}
\begin{proof}
Induction on $\deg(f)$. Suppose that $f$ has degree $1$, say, $f(x)=\alpha x+\beta$. Then for any $\gamma\in \Gamma_K$, there is $d_0\in D$ such that $v(f(x_d))>\gamma$ for all $d_0<d\in D$. Now $v(\alpha x_d+\beta)>\gamma$ implies that $v(x_d-(-\frac{\beta}{\alpha}))>\gamma-v(\alpha)$. So
\[{\lim_{{d\in D,d \to +\infty}}}|x_d-(-\frac{\beta}{\alpha})|=0\]
 and hence
\[{\lim_{{d\in D, d \to +\infty}}}x_d=-\frac{\beta}{\alpha}\ \text{and}\ f(-\frac{\beta}{\alpha})=0\]
as required.

Now suppose that $\deg(f)=n+1>1$. We see that the derivative $f'$ has degree $n$.

If there are $\gamma_0\in \Gamma_K$ and $\varepsilon_0\in D$ such that $v(f'(x_\varepsilon))\leq \gamma_0$ for all  $\varepsilon_0<\varepsilon\in D$. Take $\varepsilon_0$ sufficiently large such that
\[
v(f(x_\varepsilon))> 4\gamma_0+1
 \]
 for all $\varepsilon_0<\varepsilon\in D$.
 Then, by Hensel's Lemma,  we see that for all $\varepsilon>\varepsilon_0$, there is $\hat x_\varepsilon$ such that
\[v(\hat x_\varepsilon-x_\varepsilon)\geq \frac{v(f(x_\varepsilon))-1}{2} \ \text{and}\ f(\hat x_\varepsilon)=0\]
As $f$ has at most finitely many roots, there is a cofinal subset $I\sq D$ and some $\hat x_\varepsilon\in K$ such that
 \[
v(\hat x_\varepsilon-x_i)>\frac{v(f(x_i))-1}{2}
  \]
  for all $i\in I$. Since $v(f(x_i))\rightarrow +\infty$, we see that $v(\hat x_\varepsilon-x_i)\rightarrow +\infty$. Thus we have
\[{\lim_{{i\in I,i \to +\infty}}}x_i=\hat x_\varepsilon\ \text{and}\  f(\hat x_\varepsilon)=0,\]
as required.

Otherwise, if for every $\gamma\in \Gamma_K$, there is $\gamma<d_\gamma\in D$ such that $v(f'(x_{d_\gamma}))> \gamma$. Then there is a cofinal subset $I=\{d_\gamma|\ \gamma\in \Gamma_K\}\sq D$ such that
\[{\lim_{{i\in I,i \to +\infty}}}f'(x_i)=0,\]
Then, by induction hypothesis, there exist a cofinal subset $J\sq I$ and  $b\in K$ such that
\[{\lim_{{j\in J,j \to +\infty}}}x_j=b \ \text{and}\  f'(b)=0.\]
Since $f$ is continuous, ${\lim_{{j\in J,j \to +\infty}}}f(x_j)=f(b)$. Now $J$ is cofinal in $I$, and $I$ is cofinal in $D$, we conclude that $J$ is cofinal in $D$. This complete the proof.
\end{proof}

\begin{Prop}\label{Prop-K is closed in E}
If $K$ is a henselian field, and $E$ is a finite extension of $K$. Then for any $\alpha\in E\backslash K$, there is $\gamma_0\in \Gamma_K$ such that $v(\alpha-a)<\gamma_0$ for all $a\in K$. Namely, $K$ is closed in $E$.
\end{Prop}
\begin{proof}
By (\cite{E-P-book}, Lemma 4.1.1), the valuation of $K$ extends  uniquely to $E$. For each $\beta\in E$, Let $g(x)=x^n+a_{n-1}x^{n-1}+...+a_1x+a_0$ be the minimal polynomial of $\beta$ over $K$, then the valuation of $\beta$ is exactly $\frac{v(a_0)}{n}$ (See \cite{F.Q.G-book}, Prop. 5.3.4).

Let $\alpha\in E\backslash K$, and $d(x)$ be the minimal polynomial of $\alpha$ over $K$ with degree $k$. Then $d(x+a)$ is the minimal polynomial of $(\alpha-a)$ over $K$ for any $a\in K$. Since $d(x+a)=xf(x)+d(a)$ for some $f(x)\in K[x]$, we see that $v(\alpha-a)=\frac{v(d(a))}{k}$. We claim that there is $\gamma_0\in \Gamma_K$ such that $v(d(a))<\gamma_0$ for all $a\in K$. Otherwise, we will find a sequence $\{a_\gamma|\  \gamma\in \Gamma_K\}$ such that $v(d(a_\gamma))>\gamma$. Replace $d(x)$ by $\epsilon d(x)$ with some $\epsilon$ sufficiently close to $0$, we may assume that $d\in R[x]$. Moreover, fix $\gamma_0\in \Gamma$, if $v(\alpha-a)>\gamma_0$,  and $v(\alpha-b)>2\gamma_0$, then $v(a-b)\geq \gamma_0$. So
\[
\{b\in K|\ v(\alpha-b)>\gamma_0\}\sq \delta_0 R
\]
for some $\delta_0\in K$, and hence
\[
\{b\in K|\ v(d(b))>\gamma_0\}=\{b\in K|\ kv(\alpha-b)>\gamma_0\}\sq k\delta_0 R.
 \]
Let $\delta=k\delta_0$. If $\delta\in R$, then, by Lemma \ref{lim is 0 implies roots exist}, there is $b\in K$ such that $d(b)=0$. However $d$ is minimal polynomial of degree $>1$, so has no roots in $K$. A contradiction.

If $\delta\notin R$, then $\delta^{-1}\in R$. Suppose that
\[
d(x)=d_kx^k+...+d_1x+d_0.
\]
Let
\[
h(x)=d_kx^k+...+d_1\delta^{-k+1}x+d_0\delta^{-k}.
\]
We see that  $h(x)\in R[x]$ and
\[
h(\delta^{-1}x)=d_k(\delta^{-1}x)^k+...+d_1\delta^{-k+1}(\delta^{-1}x)+d_0\delta^{-k}=\delta^{-k}d(x).
\]
Now we have
\[
v(h(\delta^{-1}a_\gamma))=v(\delta^{-k}d(a_\gamma))>\gamma-kv(\delta).
\]
For $\gamma>\gamma_0$, we have $a_\gamma\in \delta R$. Therefore $\delta^{-1} a_\gamma\in R$ for all $\gamma>\gamma_0$.
Applying Lemma \ref{lim is 0 implies roots exist} to $h(x)$, we can find $c\in K$ such that
\[h(c)=h(\delta^{-1}\delta c)=0=\delta^{-k}d(\delta c).\]
So $d(\delta c)=0$. A contradiction.
\end{proof}

Now we assume that  $K$ is an elementary extension of $\Q$ in the language of rings $L_r$. This follow result was proven by \cite{Scowcroft-van den Dries} in the case of $K=\Q$.
\begin{Lemma}\label{partiton-of-definable-functions}
Let $\bar x=(x_1,...,x_m)$ and $f(\bar x,y)=\sum_{i=0}^n p_i(\bar x)y^i\in K[\bar x, y]$. Then there is a partition of
\[
R=\{\bar x\in K^m|\ \bigvee_{i=0}^{n}p_i(\bar x)\neq 0\wedge \exists y (f(\bar x,y)=0)\ \}
\]
into finitely many definable subsets $S$, over each of which $f$ has some fixed number $k\geq 1$ of distinct roots in $K$ with fixed multiplicities $m_1,...,m_k$. For any fixed $\bar x_0\in S$,  let the roots of $f(\bar x_0, y)$ be $r_1,...,r_k$, and $e=\max\{v(r_i-r_j)|\ 1\leq i<j\leq k\}$. Then $\bar x_0$  has a neighborhood $N\sq K^m$, $\gamma\in \Gamma_K$, and continuous, definable functions $F_1,...,F_k: S\cap N\longrightarrow K$ such that for each $\bar x\in S\cap N$,  $F_1(\bar x),...,F_k(\bar x)$ are roots of $f(\bar x, y)$ of multiplicities $m_1,...,m_k$ and $v(F_i(\bar x)-r_i)>2e$.
\end{Lemma}
\begin{proof}
The proof of Lemma 1.1 in \cite{Scowcroft-van den Dries}  applies almost word for word to the present context. The only problem is that the authors used Fact \ref{Qp is closed in E} in their proof. But the Proposition \ref{Prop-K is closed in E} saying that we could replace $\Q$ by arbitrary $K\models \Th(\Q)$ in our argument.
\end{proof}

\begin{Rmk}
Lemma 1.1 of \cite{Scowcroft-van den Dries} saying that  definable functions $F_1,...,F_k$ are not only continuous but analytic. However we can't proof it in arbitrary $K\models \Th(\Q)$ as $K$ might not be complete as a topological field.
\end{Rmk}

Similarly, Lemma 1.3 in \cite{Scowcroft-van den Dries} could be generalized to arbitrary $K\models \Th(\Q)$ as follows:
\begin{Lemma}
If $A\sq K^m$ and $f:A\longrightarrow K$ is definable. Then there is a definable set $B\sq A$, open in $K^m$ such that $A\backslash B$ has no interior and $f$ is continuous on $B$.
\end{Lemma}
\begin{proof}
The proof of Lemma 1.3 in \cite{Scowcroft-van den Dries}  applies almost word for word to the present context.
\end{proof}

\subsection{Dimensions}

We now assume that $K$ is an elementary extension of $\Q$.

\begin{Lemma}\label{dim(X)=dim(f(X))}
Suppose that $A\sq K$, $X,Y$ are $A$-definable in $K$,  $f:X\longrightarrow Y$ is an $A$-definable function. If $f$ is a finite-to-one map, $\dim(X)=\dim(f(X))$.
\end{Lemma}
\begin{proof}
Let $\K$ be a saturated elementary extension of $K$. By  Fact \ref{alg-dimension and topological dimension} (iii), there is $r\in \N$ such that $|f^{-1}(y)|\leq r$ for all $y\in Y(\K)$.  For any  $a\in X(\K)$,
 since
 \[
 |\{b\in X|\ f(b)=f(a)\}|\leq r,
 \]
  we see that $a\in \acl(A, f(a))$. So $\dim(a/A,f(a))=0$. By Fact \ref{alg-dimension and topological dimension} (i) we have
\[
\dim(a/A)=\dim(a,f(a)/A)=\dim(a/A,f(a))+\dim(f(a)/A)=0+\dim(f(a)/A).
\]
So $\dim (a/A)=\dim(f(a)/A)$. By Fact\ref{alg-dimension and topological dimension} (ii), we conclude that $\dim(X)=\dim(Y)$.
\end{proof}

\begin{Lemma}\label{dim(X)>dim(f(X))}
Suppose that $A\sq K$,  $f:X\longrightarrow Y$ is an $A$-definable function in $K$. Then
\[\dim(X)\geq\dim(f(X)).\]
\end{Lemma}
\begin{proof}
Generally, we have
\[
\dim(a/A)=\dim(a,f(a)/A)=\dim(a/A,f(a))+\dim(f(a)/A)\geq \dim(f(a)/A).
\]
By Fact\ref{alg-dimension and topological dimension} (ii), we conclude that $\dim(X)\geq\dim(Y)$.
\end{proof}

\begin{Cor}\label{dim(X)=dim(f(X))-ii}
Suppose that $A\sq K$,  $f:X\longrightarrow Y$ is an $A$-definable bijection function in $K$. Then
\[\dim(X)=\dim(f(X)).\]
\end{Cor}
\begin{proof}
$f^{-1}$ is a definable function as $f$ is bijection. So we conclude that
\[
\dim(X)\geq \dim(f(X))=\dim(Y)\geq \dim(f^{-1}(Y))=\dim(X).
\]
\end{proof}

\begin{Lemma}\label{dim(X,Y)=max{dim(X),dim(Y)}}
Suppose that  $X,Y$ are $A$-definable in $K$. Then
\[\dim(X\cup Y)=\max\{\dim(X),\dim(Y)\}.\]
\end{Lemma}
\begin{proof}
By Fact\ref{alg-dimension and topological dimension} (ii).
\end{proof}

\begin{Lemma}\label{finite-to-one-dim(Y)=dim(X)}
Let $X\sq K^n$. Then $\dim(X)$ is the minimal $k\leq n$ such that  there is definable $Y\sq X$ with $\dim(Y)=\dim(X)$ and projection
\[
\pi:X\longrightarrow K^k;\ \ (x_1,...,x_n)\mapsto (x_{r_1},...,x_{r_k})
\]
is a finite-to-one map on $Y$, for  suitable  $1\leq r_1<...<r_k\leq n$.
\end{Lemma}
\begin{proof}
Let $k$ be as above and  $\pi:X\longrightarrow K^k$ be a projection with $\pi(x_1,...,x_n)=(x_{r_1},...,x_{r_k})$. If $Y\sq X$ such that the restriction $\pi\upharpoonright Y: Y\longrightarrow K^k$ is a finite-to-one map. Then by Lemma \ref{dim(X)=dim(f(X))} we have $\dim(Y)=\dim(\pi(Y))$ and hence
\[
\dim(X)=\dim(Y)=\dim(\pi(Y))\leq k.\]

Now suppose  that $\dim(X)=l\leq k$. Without loss of generality, we assume that $f:X\longrightarrow K^l$; $(x_1,...,x_n)\mapsto (x_1,...,x_l)$ is a projection such that $f(X)$ has nonempty interior. We claim that
\begin{Claim}
Let $Z_0=\{b\in K^{l}|\ f^{-1}(b) \ \text{is finite}\ \}$ and $Z_1=K^{l}\backslash Z_0$. Then $\dim(Z_1)<l$.
\end{Claim}
\begin{proof}
Clearly,
\[
Z_1=\{b\in K^{l}|\ \dim(f^{-1}(b))\geq 1\}
 \]
 is definable in $K$.  If $\dim(Z_1)=l$. Then, there is $\beta\in Z_1(\K)$ such that $\dim(\beta/A)=l$, where is $\K\succ K$ is saturated. Since $\dim(f^{-1}(\beta))\geq 1$, by Fact \ref{alg-dimension and topological dimension} (ii), there is $\alpha\in \dim(f^{-1}(\beta))$ such that $\dim(\alpha/A,\beta)\geq 1$. By Fact \ref{alg-dimension and topological dimension} (i), we conclude that
\[
\dim(\alpha/A)=\dim(\alpha,f(\alpha)/A)=\dim(a/A,f(\alpha))+\dim(f(\alpha)/A)\geq l+1.
\]
But $\dim(\alpha/A)\leq \dim(X)=l$. A contradiction.
\end{proof}
Since $\dim(Z_1)<l$, by Lemma \ref{dim(X,Y)=max{dim(X),dim(Y)}}, $\dim(Z_0)=l$. The restriction of $f$ on $f^{-1}(Z_0)$ is a finite-to-one map, we conclude that
\[
\dim(f^{-1}(Z_0))=\dim(Z_0)=l=\dim(X)
\]
by Lemma  \ref{dim(X)=dim(f(X))}. Now $\dim(f^{-1}(Z_0))=\dim(X)$ and the restriction of $f$ on $f^{-1}(Z_0)$ is a finite-to-one map. So $k\leq l$ as $k$ is minimal. We conclude that $k=l=\dim(X)$ as required.
\end{proof}

\begin{Cor}\label{finite-partition-finite-to-one}
Let $X\sq K^n$ be definable with $\dim(X)=k$. Then there exists a  partition of $X$ into finitely many $K$-definable subsets $S$ such that whenever $\dim(S)=\dim(X)$, there is a  projection $\pi_S: S\longrightarrow K^k$ on $k$ suitable coordinate axes which is finite-to-one.
\end{Cor}
\begin{proof}
Let $X_0=X$ and $[n]^k$ be the set of all subset of $\{1,...,n\}$ of cardinality $k$. By Lemma \ref{finite-to-one-dim(Y)=dim(X)}, there exist  $D_0=\{r_1,...,r_k\}\in [n]^k$,   $S_0\sq X$ with $\dim (S_0)=\dim(X_0)$, such that the projection
\[
\pi:(x_1,...,x_n)\mapsto (x_{r_1},...,x_{r_k})
\]
is  finite-to-one on $S_0$ and infinite-to-one on $X_0\backslash S_0$.
If $\dim(X_0\backslash S_0)<\dim(X_0)$, then the partition $\{X_0\backslash S_0, S_0\}$ meets our requirements.

Otherwise, let $X_1=X_0\backslash S_0$, we could find $D_1\in [n]^k\backslash \{D_0\}$ and $S_1\in X_1$  such that the projection on coordinate axes from $D_1$ is finite-to-one over $S_1$.
Repeating the above steps, we obtained  sequences $X_i$ and $S_i$ such that $X_{i+1}=X_i\backslash S_i$. As $[n]^k$ is finite, there is a minimal $t\in \N$ such that  $\dim(X_t)<\dim(X_0)$ and $\dim(S_i)=\dim(X_0)$ for all $i<t$. It is easy to see that $\{S_0,...,S_{t-1}, X_t\}$ meets our requirements.
\end{proof}

Recall that by \cite{L-Dries Skolem}, $\Th(\Q)$ admits definable Skolem functions.  Namely, we have
\begin{Fact}\cite{L-Dries Skolem}\label{Skolem-functions}
Let $A\sq K$ and $\phi(\bar x, y)$ be a $L_A$-formula such that
\[
K\models \forall \bar x\exists y\phi(\bar x,y).
\]
Then there $A$-definable function $f: K^m\rightarrow K$ such that $K\models \forall \bar x\phi(\bar x,f(\bar x))$.
\end{Fact}

With the above Fact, we could refine Corollary \ref{finite-partition-finite-to-one} as follows:
\begin{Cor}\label{1-1-projection}
Let $X\sq K^n$ be definable with $\dim(X)=k$. Then there exists a  partition of $X$ into finitely many $K$-definable subsets $S$ such that whenever $\dim(S)=\dim(X)$, there is a  projection $\pi_S: S\longrightarrow K^k$ on $k$ suitable coordinate axes which is   injective.
\end{Cor}
\begin{proof}
Let $X_0=X$. By Corollary \ref{finite-partition-finite-to-one}, we may assume that the projection $\pi:X_0 \longrightarrow K^k$  given by $(x_1,...,x_n)\mapsto (x_1,...,x_k)$ is finite-to-one. By compactness, there is $r\in \mathbb N$ such that
\[
|\pi^{-1}(\bar y)\cap X_0|\leq r
\]
 for all $\bar y\in K^k$.
Induction on $r$. If $r=1$, then $\pi$ is injective on $X_0$. Otherwise, by Fact \ref{Skolem-functions}, there is a definable function
\[
f:\pi(X)\longrightarrow X
\]
such that $\pi(f(\bar y))=x$ for all $\bar y\in \pi(X_0)$. It is easy to see that $f$ is injective and hence, by Corollary \ref{dim(X)=dim(f(X))},  $S_0=f(\pi(X_0))$ is a definable subset of $X$ of dimension $k$. Moreover $\pi:S_0\longrightarrow K^k$ is exactly the inverse of $f$, hence injective. If $\dim(X_0\backslash S_0)$ then the partition $\{X_0\backslash S_0, S_1\}$ satisfies our require. Otherwise, $X_1=X_0\backslash S_0$ has dimension $k$ and
\[
|\pi^{-1}(\bar y)\cap X_1|\leq r-1
\]
By our induction hypothesis, there is a partition of $X_1$ into finitely may definable subsets  meets our requirements. This completes the proof.
\end{proof}

\begin{Thm}
Let $B\sq K^m$ be definable in $K$. Then $\dim_K(B)\geq \dim_{\Q}(B\cap \Q^m)$.
\end{Thm}
\begin{proof}
Suppose that $\dim_K(B)=k$. By Lemma \ref{dim(X,Y)=max{dim(X),dim(Y)}} and  Corollary \ref{1-1-projection}, we may assume that $\pi: B\longrightarrow K^k$  is injective. The restriction of $\pi$ to $B\cap \Q^m$ is a injective projection from $B\cap \Q^m$ to $\Q^k$. By Lemma \ref{finite-to-one-dim(Y)=dim(X)}, $\dim_{\Q}(B\cap \Q^m)\leq k$.
\end{proof}

Note that $P_n(K)=\{a\in K|\ a\neq 0\wedge \exists b\in K(a=b^n)\}$ is an open subset of $K$ whenever $K$ is a hensilian field. For any polynomial $f(x_1,...,x_m)\in K[x_1,...,x_m]$,
\[
P_n(f(K^m))=\{a\in K^m|\ f(a)\neq 0\wedge \exists b\in K(f(a)=b^n)\}
\]
is an open subset of $K^m$ since $f$ is continuous.

\subsection{Standard Part Map and Definable Functions}
The following Facts will be used later.

\begin{Fact}\cite{DELON}\label{Every-type-over-QP-is-definable}
Every complete $n$-type over $\Q$ is definable. Equivalently, for any $K\succ \Q$, any $L_r$-formula $\phi(x_1,...,x_n, y_1,...,y_m)$, and any $\bar b\in K^{m}$, the set
\[
\{\bar a\in \Q^{n}|\ K\models \phi(\bar a,\bar b)\}
\]
is definable in $\Q$.
\end{Fact}

\begin{Fact}\label{O-P-partition-I}\cite{O-P}
Let $X\sq K^m$ be a $\Q$-definable open set, let $Y\sq X$ be a $K$-definable subset of $X$. Then either $Y$ or $X\backslash Y$ contains a $\Q$-definable open set.
\end{Fact}

\begin{Fact}\label{O-P-st-image-has-no-interior}\cite{O-P}
Let $X\sq K^m$ be a $K$-definable set. Then $\st(X)\cap\st(K^m\backslash X)$ has no interior.
\end{Fact}

Recall that $\mu$ is the collection of all infinitesimals of $K$ over $\Q$, which induces a equivalence relation $\backsim_\mu$ on $K$, which is defined by
\[
a \backsim_\mu b\ \iff \ a-b\in \mu.
\]

\begin{Def}
Let $f(\bar x,y),g(\bar x,y)\in K[\bar x,y]$ be polynomials. By $f \backsim_\mu g$ we mean that
\begin{enumerate}
  \item [(i)] if $|\bar x|=0$, $f(y)=\sum_{i=1}^na_iy^i$, and $g(y)=\sum_{i=1}^nb_iy^i$, then $f \backsim_\mu g$ iff $a_i \backsim_\mu b_i$ for each $i\leq n$.
  \item  [(ii)]  if $|\bar x|>0$, $f(\bar x,y)=\sum_{i=1}^na_i(\bar x)y^i$, and $g(y)=\sum_{i=1}^nb_i(\bar x)y^i$,  where $a$'s and $b$'s are polynomials with variables from $\bar x$. Then $f \backsim_\mu g$ iff $a_i \backsim_\mu b_i$ for each $i\leq n$.
\end{enumerate}
\end{Def}

\begin{Lemma}\label{mu-equi-on-poly}
Let $\bar x=(x_1,...,x_n)$, $f(\bar x)$ and $g(\bar x)$ be polynomials over $K$ with $f \backsim_\mu g$. If $\bar a=(a_1,...,a_n)$ and $\bar b=(b_1,...,b_n)$ are tuples from $K$ with $a_i \backsim_\mu b_i$ for each $i\leq n$. Then $f(\bar a) \backsim_\mu g(\bar b)$.
\end{Lemma}
\begin{proof}
We see that $\alpha\in \mu$ iff $v(\alpha)>\mathbb Z$. Since $v(\alpha+\beta)\geq \min\{v(\alpha),v(\beta)\}$ and $v(\alpha \beta)=v( \alpha)+v(\beta)$, we see that  $\mu$ is closed under addition and multiplication. As polynomials are functions obtained by compositions of addition and multiplication, we conclude that $f(\bar a) \backsim_\mu g(\bar b)$.
\end{proof}

Since $V\sq K$ is also closed under addition and multiplication. We conclude directly that:
\begin{Cor}\label{st(f)(st(s))=st(f(a))}
Let $\bar x=(x_1,...,x_n)$, and $f(\bar x)\in K[\bar x]$ be a polynomial with every coefficient contained in $V$. If $a=(a_1,...,a_n)\in V^n$, then $\st(f)(\st(a))=\st(f(a))$.
\end{Cor}

\begin{Cor}\label{f(b)=0-implise-b-in-V}
Let  $f(x)=a_nx^n+...+a_1x^1+a_0$ be a polynomial over $K$ with every coefficient contained in $V$ and $a_n\notin \mu$. If $b\in K$ such that $f(b)=0$. Then $b\in V$.
\end{Cor}
\begin{proof}
Suppose for a contradiction that $b\notin V$.  Then $\st(b^{-1})=0$. Clearly, we have
\[
b^{-n}f(b)=a_n+...+a_1b^{-n+1}+a_0b^{-n}=0.
\]
Let
\[g(y)=a_0y^n+...+a_{n-1}y+a_n.\]
Then $g(b^{-1})=0$. By Corollary \ref{st(f)(st(s))=st(f(a))}, we have $\st(g)(\st(b^{-1}))=0$.  As $\st(b^{-1})=0$, we see that $\st(a_n)=0$,  which contradicts to $a_n\notin \mu$.
\end{proof}

\begin{Fact}\label{St-part-image-definable}\cite{O-P}
Let $S\sq K^m$ be definable in $K$. Then $st(S\cap V^m)\sq \Q^m$ is definable in $\Q$.
\end{Fact}

\begin{Lemma}\label{Lemma-definable-partition-I}
Let $f:K^k\longrightarrow K$ be definable in $K$. Then
\begin{itemize}
  \item [(i)] $X_{\infty}=\{a\in \Q^k|\ f(a)\notin V\}$ is definable in $\Q$.
  \item [(ii)] Let $X=\Q^k\backslash X_{\infty}$. Then $g: X\longrightarrow \Q$ given by $a\mapsto \st(f(a))$ is definable in $\Q$
\end{itemize}
\end{Lemma}
\begin{proof}
By Fact \ref{Every-type-over-QP-is-definable}, there is a $L_{\Q}$-formula $\phi(x,y)$ such that for all $a\in \Q^k$ and $b\in \Q$, we have
\[
\Q\models \phi(a,b)\  \iff \ v(f(a))<v(b).
\]
Hence
\[
a\in X_\infty \iff \Q\models \forall y\phi(a,y),
\]
which shows that $X_\infty$ is definable in $\Q$. Again by Fact \ref{Every-type-over-QP-is-definable}, there is $L_{\Q}$-formula $\psi(x,y_1,y_2)$ such that for all $a\in \Q^k$, $b_1,b_2\in\Q$,
\[
M\models \psi(a,b_1,b_2) \iff v(f(a)-b_1)>v(b_2-b_1).
\]
Therefore
\[
b=\st(f(a))\ \iff\ \Q\models \forall y_1\forall y_2(v(b-y_1)>v(y_1-y_2)\rightarrow \psi(a,y_1,y_2)).
\]
for all $a\in \Q^k$ and $b\in \Q$. We conclude that $g: X\longrightarrow \Q$, $a\mapsto \st(f(a))$ is definable in $\Q$
\end{proof}

\begin{Lemma}\label{a-notin-X(K)}
Let $X\sq \Q^m$ be a clopen subset of $\Q^m$. If $\bar a\in V^m$ and $\st(\bar a)\notin X$, then $\bar a\notin X(K)$.
\end{Lemma}
\begin{proof}
As $X$ is clopen and $\st(\bar a)\notin X$, there is $N\in \mathbb Z$ such that
\[
B_{(\bar a, N)}=\{\bar b\in \Q^m|\ \bigwedge_{i=1}^{m} v(b_i-\st(a_i))>N\}\cap X=\emptyset.
\]
So $B_{(\bar a, N)}(K)\cap X(K)=\emptyset$. But  $v(a_i-\st(a_i))>\mathbb Z$, hence
\[
\bar a\in \{\bar b\in K^m|\ \bigwedge_{i+1}^{m} v(b_i-\st(a_i))>N\}=B_{(\bar a, N)}(K).
\]
So $\bar a\notin X(K)$  as required.
\end{proof}

\begin{Lemma}\label{graph-in-variety}
If $X\sq K^m$ and $f:X\longrightarrow K$ are definable in $K$, then there is a polynomial $q(x_1,...,x_m,y)$ such that the graph of $f$ is contained in the variety
\[
\{(a_1,...,a_m,b)\in K^{m+1}|\ q(a_1,...,a_m,b)=0\}.
\]
\end{Lemma}
\begin{proof}
Let $Y$ be the graph of $f$. Since $\Th(\Q)$ has quantifier elimination, $Y$ is defined by a disjunction $\bigvee_{i=1}^{s}\phi_{i}(\bar x)$, where each $\phi_{i}(\bar x)$ is a conjunction
\[
(\bigwedge_{j=1}^{l_{i}}g_{i_j}(\bar x,y)=0)\ \wedge \ \bigwedge_{j=1}^{l_{i}}P_{n_{{i}_j}}(h_{{i}_j}(\bar x, y)),
\]
where $g$'s and $h$'s belong to $K[\bar x,y]$. Now each  $P_{n_{{i}_j}}(h_{{i}_j}(\bar x, y))$ defines an open subset of $K^{m+1}$. Since $\dim(Y)\leq m$, we see that for each $i\leq s$, there is $f(i)\leq l_i$ such that $g_{i_{f(i)}}\neq 0$ . Let $q(\bar x,y)=\Pi_{i=1}^{s}g_{i_{f(i)}}(\bar x,y)$. Then
\[
Y\sq \{(a_1,...,a_m,b)\in K^{m+1}|\ q(a_1,...,a_m,b)=0\}
\]
as required.
\end{proof}

\begin{Prop}\label{Prop-no-interior}
If $f:K^m\longrightarrow K$ is definable in $K$. Let $X=\{\bar a\in \Q^m|\ f(\bar  a)\in V\}$.
Then
\[
D_X=\{\bar a\in X|\ \exists \bar b,\bar c\in\st^{-1}(\bar a)\bigg(f(\bar b)-f(\bar c)\notin \mu\bigg)\}.
\]
has no interiors.
\end{Prop}
\begin{proof}
By Lemma \ref{graph-in-variety}, there is a polynomial
\[g(x_1,...,x_m,y)\in K[x_1,...,x_m,y]\]
such that the graph of $f$ is contained in the variety of $g$. Without loss of generality, we may assume that each coefficient of $g$ is in $V$, otherwise, we could replace $g$ by $g/c$, where $c$ is a coefficient of $g$ with minimal valuation. Moreover, we could assume that at least one coefficient of $g$ is not in $\mu$.

Suppose for a contradiction that $D_X$ contains a open subset of $\Q^m$. Shrink $D_X$ if necessary, we may assume that $D_X\sq X$ is a $\Q$-definable open set in $\Q^m$. By Lemma \ref{partiton-of-definable-functions},  there is a partition $\cal P$ of $D_X(K)\sq K^m$ into finitely many definable subsets $S$, over each of which $g$ has some fixed number $k\geq 1$ of distinct roots in $K$ with fixed multiplicities $m_1,...,m_k$. For any fixed $\bar x_0\in S$,  let the roots of $g(\bar x_0, y)$ be $r_1,...,r_k$, and $e=\max\{v(r_i-r_j)|\ 1\leq i<j\leq k\}$. Then $\bar x_0$  has a neighborhood $N\sq K^m$, $\gamma\in \Gamma_K$, and continuous, definable functions $F_1,...,F_k: S\cap N\longrightarrow K$ such that for each $\bar x\in S\cap N$,  $F_1(\bar x),...,F_k(\bar x)$ are roots of $g(\bar x, y)$ of multiplicities $m_1,...,m_k$ and $v(F_i(\bar x)-r_i)>2e$.

Since $D_X(K)$ is a $\Q$-definable open subset of $X(K)$. By Fact \ref{O-P-partition-I}, some $S\in \cal P$ contains a $\Q$-definable open subset $\psi(K^m)$ of $X(K)$. Where $\psi$ is an $L_{\Q}$-formula. Let $A_0=\phi(\Q^m)$. Then $A_0\sq A$ is an open subset of $\Q^m$, and over $A_0(K)$ we have
\begin{itemize}
  \item [(i)] $g$ has some fixed number $k\geq 1$ of distinct roots in $K$ with fixed multiplicities $m_1,...,m_k$.
  \item [(ii)] For any fixed $\bar x_0\in A_0(K)$,  let the roots of $g(\bar x_0, y)$ be $r_1,...,r_k$, and $e=\max\{v(r_i-r_j)|\ 1\leq i<j\leq k\}$. Then $\bar x_0$  has a neighborhood $N\sq K^m$, $\gamma\in \Gamma_K$, and continuous, definable functions $F_1,...,F_k: A_0(K)\cap N\longrightarrow K$ such that for each $\bar x\in A_0(K)\cap N$,  $F_1(\bar x),...,F_k(\bar x)$ are roots of $g(\bar x, y)$ of multiplicities $m_1,...,m_k$ and $v(F_i(\bar x)-r_i)>2e$.
  \item [(iii)] for any $\bar a\in A_0$, there exist $\bar b,\bar c\in\st^{-1}(\bar a)$ such that $\st(f(\bar b))\neq \st(f(\bar c))$.
\end{itemize}
Suppose that
\[g(x_1,...,x_m,y)=\sum_{i=0}^{n}g_i(x_1,...,x_m)y^i,\]
where each $g_i(\bar x)\in K[\bar x]$. Since the variety $\{\bar a\in \Q^m|\ \st(g_n)(\bar a)=0\}$ has dimension $m-1$, it has no interior. By Fact \ref{O-P-partition-I},
\[
A_0\backslash \{\bar a\in \Q^m|\ \st(g_n)(\bar a)=0\}
\]
contains a open subset of $\Q^m$. Without loss  of generality, we may assume that
\[\{\bar a\in \Q^m|\ \st(g_n)(\bar a)=0\}\cap A_0=\emptyset.\]
Since the family of clopen subsets forms a base for topology on $\Q^m$, we may assume that $A_0$ is clopen.
We now claim that
\begin{Claim1}
For every $\bar a\in A_0(K)$, $g_n(\bar a)\notin \mu$.
\end{Claim1}
\begin{proof}
Otherwise, by Corollary \ref{st(f)(st(s))=st(f(a))}, we have $\st(g_n)(\st(\bar a))=\st(g_n(\bar a))=0$. So $\st(\bar a)\notin A_0$. By Lemma \ref{a-notin-X(K)}, we see that $\bar a\notin A_0(K)$. A contradiction.
\end{proof}
By Claim 1 and Corollary \ref{f(b)=0-implise-b-in-V}, we see that for every $\bar a\in A_0(K)$  and $b\in K$, if $g(\bar a, b)=0$, then $b\in V$. By Corollary \ref{st(f)(st(s))=st(f(a))}, we conclude the following claim
\begin{Claim2}
For every $\bar a\in A_0(K)$ and $b\in K$, if $g(\bar a, b)=0$, then $b\in V$ and
\[\st(g)(\st(\bar a), \st(b))=0.\]
\end{Claim2}
Now $\st(g)$ is a polynomial over $\Q$. Applying Lemma \ref{partiton-of-definable-functions} to $\st(g)$  and  Fact \ref{O-P-partition-I}, and shrink $A_0$ if necessary, we may assume that
\begin{itemize}
  \item $\st(g)$ has some fixed number $d\geq 1$ of distinct roots in $K$ with fixed multiplicities $n_1,...,n_d$ over $A_0$.
  \item Fix some $\bar x_0\in A_0$,  let the roots of $\st(g)(\bar x_0, y)$ (in $\Q$) be $s_1,...,s_d$, and
      \[
      \Delta=\max\{v(s_i-s_j)|\ 1\leq i<j\leq d\}.
      \]
      Then there are  definable continuous functions $H_1,...,H_d: A_0 \longrightarrow \Q$ such that for each $\bar x\in A_0$,  $H_1(\bar x),...,H_d(\bar x)$ are roots of $\st(g)(\bar x, y)$ of multiplicities $n_1,...,n_d$ and $v(H_i(\bar x)-s_i)>2\Delta$.
  \item for any $\bar a\in A_0$, there exist $\bar b,\bar c\in\st^{-1}(\bar a)$ such that $\st(f(\bar b))\neq \st(f(\bar c))$.
\end{itemize}
By Claim 2, we see that for any $\bar x\in A_0(K)$, and $b\in K$, if $g(\bar x,b)=0$, then $b\backsim_\mu H_i(\st(\bar x))$ for some $i\leq d$. As $g(\bar x,f(\bar x))=0$ for all $\bar x\in K^m$, we see that
\begin{Claim3}
For each $\bar x\in A_0(K)$, $f(\bar x)\backsim_\mu H_i(\st(\bar x))$ for some $i\leq d$.
\end{Claim3}
Let $D_i=\{\bar x\in A(K)|\ v(f(\bar x)-s_i)>2\Delta\}$. We claim that
\begin{Claim4}
$A_0(K)=\bigcup_{i=1}^{d} D_i$ and $D_i\cap D_j=\emptyset$ for each $i\neq j$. Namely, $\{D_1,...,D_d\}$ is a partition of $A_0(K)$.
\end{Claim4}
\begin{proof}

Let $\bar x\in A_0(K)$. By Claim 3, there is some $i\leq d$ such that $f(\bar x)\backsim_\mu H_i(\st(\bar x))$. It is easy to see that
\[
v(f(\bar x)-s_i)=v(f(\bar x)-H_i(\st(\bar x))+H_i(\st(\bar x))-s_i)=v(H_i(\st(\bar x))-s_i)>2\Delta.
\]
So $\bar x\in D_i$ and this implies that $A_0(K)=\bigcup_{i=1}^{d} D_i$.

On the other side, if $\bar x\in D_i\cap D_j$ for some  $1\leq i< j\leq d$, we have $v(f(\bar x)-s_i)>2\Delta$ and $v(f(\bar x)-s_j)>2\Delta$, which implies that
\[
v(s_i-s_j)=v(s_i-f(\bar x)+f(\bar x)-s_j)\geq\min\{v(s_i-f(\bar x)),v(f(\bar x)-s_j)\}\geq 2\Delta.
\]
But $v(s_i-s_j)\leq \Delta$. A contradiction.
\end{proof}

\begin{Claim5}
Let $i\leq d$. For any $\bar a, \bar b\in D_i$, if $\bar a\backsim_\mu \bar b$ then $\st(f(\bar a))=\st(f(\bar b))$.
\end{Claim5}
\begin{proof}
Let $\bar x\in D_i$. By Claim 3, there is $j\leq d$ such that $\st(f(\bar x))=H_j(\st(\bar x))$. We see that
 \[
v(f(\bar x)-s_j)=v(f(\bar x)-H_j(\st(\bar x))+H_j\st(\bar x))-s_j)=v(H_j(\st(\bar x))-s_j)>2\Delta.
 \]
So $\bar x\in D_j$. By Claim 4, $i=j$. We conclude  that $\st(f(\bar x))=H_i(\st(\bar x))$ whenever $\bar x\in D_i$.  This complete the proof of Claim 5.
\end{proof}
Recall that for any $\bar a\in A_0$, there exist $\bar b,\bar c\in\st^{-1}(\bar a)$ such that $\st(f(\bar b))\neq \st(f(\bar c))$. By Claim 4 and Claim 5, we see that for each $\bar a\in A_0$, there is $1\leq i\neq j\leq d$ such that $\bar a\in \st(D_i)\cap\st(D_j)$. This means that
\[
A_0\sq \bigcup_{1\leq i\neq j\leq d} \st(D_i)\cap\st(D_j)
\]
By Fact \ref{O-P-st-image-has-no-interior}, each $\st(D_i)\cap\st(D_j)$ has no interior. By Fact \ref{O-P-partition-I}, $A_0$ has no interiors. A contradiction.
\end{proof}

\begin{Cor}\label{Cor-no-interior}
If $f:K^m\longrightarrow K$ is definable in $K$. Let $X_\infty=\{\bar a\in \Q^m|\ f(\bar a)\notin V\}$.
Then
\[
U=\{\bar a\in X_\infty|\ \exists \bar b,\bar c\in\st^{-1}(\bar a)\bigg(f(\bar b)\in V\wedge f(\bar c)\notin V\bigg)\}.
\]
has no interior.
\end{Cor}
\begin{proof}
Otherwise, suppose that $U\sq K^m$ is open. Applying  Proposition \ref{Prop-no-interior} to $g(x)=(f(x))^{-1}$, we see that $g(U)\sq V$, and for all $\bar  a\in U$  there are $\bar b,\bar  c\in\st^{-1}(a)$ such that $\st(g(\bar b))\neq 0$ and $\st(g(\bar c))= 0$. A contradiction.
\end{proof}

\begin{Lemma}\label{Dx U are Definable}
Let $f:K^k\longrightarrow K$ be definable in $K$,  $X=\{a\in \Q^k|\ f(a)\in V\}$, and $X_\infty=\{a\in \Q^k|\ f(a)\notin V\}$. Then both
\[
D_X=\{\bar a\in X|\ \exists \bar b,\bar c\in\st^{-1}(\bar a)\bigg(f(\bar b)-f(\bar c)\notin \mu\bigg)\}.
\]
and
\[
U=\{\bar a\in X_\infty|\ \exists \bar b,\bar c\in\st^{-1}(\bar a)\bigg(f(\bar b)\in V\wedge f(\bar c)\notin V\bigg)\}.
\]
are definable sets over $\Q$
\end{Lemma}
\begin{proof}
Let $X_0=\{\bar a\in X|\ \st(f(\bar a))=0\}$ and $X_1=\{\bar a\in X|\ \st(f(\bar a))\neq 0\}$. As we showed in Lemma \ref{Lemma-definable-partition-I}, both $X_0$ and $X_1$ are $\Q$-definable sets. Let
\[
g:K^k\backslash f^{-1}(0)\longrightarrow K
\]
be the $K$-definable function given by $\bar x\mapsto 1/f(\bar x)$.
Let $Y\sq K^{k+1}$ be the graph of $f$ and  $Z\sq K^{k+1}$ be the graph of $g$. For each $\bar a\in \Q^k$, let\[\st(Y)_{\bar a}=\{b\in \Q|\ (\bar a,b)\in \st(Y)\}\ \  \text{and}\  \st(Z)_{\bar a}=\{b\in \Q|\ (\bar a,b)\in \st(Z)\}.\]
Let
\[
S_1=\{\bar a\in X|\ |\st(Y)_{\bar a}|>1\}, \  S_2=\{\bar a\in X_0|\ |\st(Z)_{\bar a}|\geq 1\},  \ \text{and}\ S_3=\{\bar a\in X_1|\ |\st(Z)_{\bar a}|> 1\}.
\]
We now show that $D_X=S_1\cup S_2\cup S_3$.

Clearly, $S_1$ and $S_3$ are subsets of $D_X$. If $\bar a\in S_2$, then $\st(f(\bar a))=0$ and there is $\bar b\in \st^{-1}(\bar a)$ such that $1/f(\bar b)\in V$, so $f(\bar a)-f(\bar b)\notin \mu$, which implies that $\bar a\in D_X$. Therefore, we conclude that $S_1\cup S_2\cup S_3\sq D_X$.

Conversely, take any $\bar a\in D_X$ and suppose that  $\bar b,\bar c\in\st^{-1}(\bar a)$ such that $f(\bar b)-f(\bar c)\notin \mu$. If both $f(\bar b)$ and $f(\bar c)$ are in $V$, then $\bar a\in S_1$; If $f(\bar b)\notin V$ and $\st(f(\bar a))=0$, then $\bar b\in\mathrm{dom}(g)$. We see that $(\bar a, 0)\in \st(Z)$, so $\bar a\in S_2$; If $f(\bar b)\notin V$ and $\st(f(\bar a))\neq0$, then $\bar a, \bar b\in \mathrm{dom}(g)$, $\st(g(\bar b))=0$ and $\st(g(\bar a))\neq 0$, which implies that  $|\st(Z)_{\bar a}|> 1$, and thus $\bar a\in S_3$. So we conclude that $D_X\sq S_1\cup S_2\cup S_3$ as required.

As $S_1$, $S_2$, and $S_3$ are definable sets over $\Q$, $D_X$ is definable over $\Q$. Similarly, $U$ is definable over $\Q$.
\end{proof}

Suppose that $C\sq \Q^m$, we define the hull $C^h$ by
\[
C^h=\{\bar x\in K^m|\ \st(\bar x)\in C\}.
\]

\begin{Thm}\label{Main-thm-ii}
Let $f:K^m\rightarrow K$ be an $K$-definable function. Then here is a finite partition ${\cal P}$ of $\Q$ into definable sets, where each set in the partition is either open in $\Q^m$ or lacks of interior. On each open set $C\in \cal P$ we have:
\begin{itemize}
\item [(i)] either $f(x)\notin V$ for all $x\in C^h$;
  \item [(ii)] or there is a continuous function $g: C\longrightarrow \Q$, definable in $\Q$, such that $f(x)\in V$ and $\st(f(x))=g(\st(x))$, for all $x\in C^h$.
\end{itemize}
\end{Thm}
\begin{proof}
Let $X,\ X_\infty$ be as in Lemma \ref{Lemma-definable-partition-I}, $D_X$ as  in Proposition \ref{Prop-no-interior}, and $U$ as in Corollary \ref{Cor-no-interior}, then $D_X$ and $U$ have no interior, and by Lemma \ref{Dx U are Definable}, they are definable. Now $\{D_X, X\backslash D_X, \ U, X_\infty\backslash U\}$ is a partition of $\Q^m$. Clearly, $\{\Int(X_\infty\backslash U), \ (X_\infty\backslash U)\ \backslash \Int(X_\infty\backslash U)\}$ is a partition of $X_\infty\backslash U$  where $\Int(X_\infty\backslash U)$ is open and $(X_\infty\backslash U)\ \backslash \Int(X_\infty\backslash U)$ lacks of interior.

Let $h:X\backslash D_X\longrightarrow \Q$ be a definable function defined by $x\mapsto \st(f(x))$.
By Theorem 1.1 of \cite{Scowcroft-van den Dries}, there is a finite partition ${\cal P}^*$  of  $X\backslash D_X$ into definable sets, on each of which $h$ is analytic. Each set in the partition is either open in $\Q^m$ or lacks of interior.

Clearly, the partition
\[
{\cal P}=\{D_X, U, \Int(X_\infty\backslash U), \ (X_\infty\backslash U)\ \backslash \Int(X_\infty\backslash U)\}\cup {\cal P}^*
\]
satisfies our condition.
 \end{proof}

We now prove our last result.

\begin{Lemma}\label{dim(st(Z))leq dim(Z)}
Let $Z\sq K^n$ be definable in $K$ of dimension $k<n$, and the projection
\[\pi:(x_1,...x_n)\mapsto (x_1,...,x_k)
\]
is  injective on $Z$. Then $\dim_{\Q}(\st(Z\cap V^n))\leq k$.
\end{Lemma}
\begin{proof}
As $\pi$ is injective on $X$, there is a definable function
\[
f=(f_1,...,f_n): K^k\longrightarrow K^n
\]
 such that
\begin{itemize}
  \item $f(\pi(\bar x))=(f_1(\pi(\bar x)),...,f_n(\pi(\bar x)))=\bar x$ for all $\bar x\in Z$;
  \item $f(\bar y)=(0,...,0)$ for all $\bar y\in K^k\backslash \pi(X)$.
\end{itemize}
By Lemma \ref{graph-in-variety}, for each $i\leq n$, there is a polynomial $F_i(\bar y,u)$ such that the graph of $f_i$ is contained in the variety
\[
V(F_i)=\{(\bar y,u)\in K^{k+1}|\ F_i(\bar y,u)=0\}
\]
of $F_i$. We assume that each coefficient belongs to $V$.
It is easy to see that for each
\[
(a_1,...,a_n)\in Z\cap V^n,
\]
 we have $f_i(\pi(a_1,...,a_n))=a_i$. So $F_i(a_1,...,a_k,a_i)=0$. By Corollary \ref{st(f)(st(s))=st(f(a))},
\[
 \st(F_i)(\st(a_1),...,\st(a_k),\st(a_i)))=0.
\]
 So $\st(Z\cap V^n)$ is contained in the variety
\[V(\st(F_1),...,\st(F_n))=\bigg\{(a_1,...,a_n)\in \Q^n|\ \bigwedge_{i=1}^n\bigg( \st(F_i)(\st(a_1),...,\st(a_k),\st(a_i)))=0\bigg)\bigg\}.\]
Let $A\sq \Q$ be the collection of all coefficients from $\st(F_i)$'s. Then for each
\[
(a_1,...,a_n)\in V(\st(F_1),...,\st(F_n)),
 \]
 we see that $a_i$ is a root of $F_i(a_1,...,a_k,u)$, and hence $a_i\in \acl(A,a_1,...,a_k)$, where $i\leq n$. This implies that
\[
\dim(a_1,...,a_n/A)=\dim(a_1,...,a_k/A)\leq k
\]
for all $(a_1,...,a_n)\in V(\st(F_1),...,\st(F_n))$. By Fact \ref{alg-dimension and topological dimension} (v), we see that
\[
\dim_{\Q}(V(\st(F_1),...,\st(F_n)))=\max\big\{\dim(a_1,...,a_n/A)|\ (a_1,...,a_n)\in V(\st(F_1),...,\st(F_n))\big\}\leq k.
\]
So $\st(Z\cap V^n)\leq k$ as required.
\end{proof}

\begin{Thm}
Let $Z\sq K^n$ be definable in $K$. Then $\dim_{\Q}(\st(Z\cap V^n))\leq \dim_K(Z)$.
\end{Thm}
\begin{proof}
Since $\st(X\cup Y)=\st(X)\cup\st(Y)$ and $\dim(X\cup Y)=\max\{\dim(X),\dim(Y)\}$ hold for all definable $X,Y\sq K^n$. Applying Corollary \ref{1-1-projection}, we many assume that $\dim(Z)=k$ and  $\pi:(x_1,...x_n)\mapsto (x_1,...,x_k)$ is  injective on $Z$. If $k=n$, then
\[\dim_{\Q}(\st(Z\cap V^n))\leq n\]
 as $\st(Z\cap V^n)\sq \Q^n$. If $k<n$, then by Lemma \ref{dim(st(Z))leq dim(Z)},
 \[\dim_{\Q}(\st(Z\cap V^n))\leq k\]
  as required.
\end{proof}

\end{document}